\documentclass[12pt]{article}

\usepackage{amsmath}
\usepackage{amsthm}
\usepackage{amssymb}
\usepackage{cite}
\usepackage{url}
\usepackage[multiple]{footmisc}
\usepackage{graphicx}
\usepackage{epstopdf}
\usepackage{chngcntr}
\usepackage{enumitem}
\usepackage{hhline}
\usepackage{array}
\usepackage{hyperref}
\usepackage{color}
\usepackage{array}
\usepackage{multirow}
\usepackage{epstopdf}   

\hypersetup{colorlinks=false}

\if@twoside
\else
\fi

\numberwithin{equation}{section}
\counterwithin{figure}{section}
\counterwithin{table}{section}

\theoremstyle{plain}
\newtheorem{theorem}{Theorem}[section]
\newtheorem{lemma}[theorem]{Lemma}

\newtheorem{algorithm}  [theorem]{Algorithm}

\theoremstyle{definition}
\newtheorem{definition}{Definition}[section]
\newtheorem{assumption}{Assumption}[section]
\newtheorem{example}{Example}[section]

\theoremstyle{remark}
\newtheorem{remark}{Remark}[section]

\newcommand{\norm}[1]{\left\|#1\right\|}

\newcommand{\abs}[1]{\left\vert#1\right\vert}
\newcommand{\spr}[1]{\left\langle\,#1\,\right\rangle}
\newcommand{\kl}[1]{\left(#1\right)}
\newcommand{\Kl}[1]{\left\{#1\right\}}

\definecolor{aog}{rgb}{0.0, 0.5, 0.0}

\newcommand{\R}{\mathbb{R}}

\newcommand{\Z}{\mathbb{Z}}

\newcommand{\yd}{y^{\delta}}

\newcommand{\m}{\text{m}}

\newcommand{\Rn}{{\R^n}}

\newcommand{\F}{\mathcal{F}}
\newcommand{\FI}{\mathcal{F}^{-1}}

\newcommand{\LtO}{{L^2(\Omega)}}

\newcommand{\LtRn}{{L^2(\R^n)}}

\newcommand{\HsO}{{H^s(\Omega)}}

\newcommand{\HsRn}{{H^s(\R^n)}}

\newcommand{\E}{{\mathcal{E}}}

\newcommand{\Js}{\mathcal{J}^s}
\newcommand{\Jms}{\mathcal{J}^{-s}}

\newcommand{\SRn}{\mathcal{S}(\Rn)}
\newcommand{\SsRn}{\mathcal{S}^*(\Rn)}

\newcommand{\HmsRn}{{H^{-s}(\Rn)}}
\newcommand{\Es}{\mathcal{E}_s}
\newcommand{\Esh}{\mathcal{E}_{s/2}}

\newcommand{\Ja}{J_{\alpha}}

\newcommand{\fD}{f^{\dagger}}
\newcommand{\fd}{f^{\delta}}
\newcommand{\fad}{f_\alpha^\delta}

\newcommand{\ga}{q_\alpha}
\newcommand{\ra}{r_\alpha}
\newcommand{\JTV}{J_{TV}}
\newcommand{\Rt}{\tilde{R}}
\newcommand{\gd}{g^{\delta}}
\newcommand{\fdTV}{\tilde{f}_{TV}^{\delta}}
\newcommand{\ft}{\Tilde{f}}
\newcommand{\ftdTV}{\tilde{f}^{\delta}_{TV}}
\newcommand{\LtSR}{L^2(\mathbb{S}\times\R)}
\newcommand{\LtRt}{L^2(\R^2)}
\newcommand{\fa}{f_{\alpha}}
\renewcommand{\fam}{f_{\alpha,m}}
\newcommand{\fak}{f_{\alpha,k}}
\newcommand{\mfz}{\mathbf{0}}
\newcommand{\Is}{\mathcal{I}_{s}}
\newcommand{\LtSnmR}{L^2(\mathbb{S}^{n-1}\times\R)}
\renewcommand{\S}{\mathcal{S}}

\title{Regularization of linear inverse problems with irregular noise using embedding operators }
\author{
Xinyan Li\footnote{School of Mathematical Sciences, Fudan University, Shanghai 200433, China, (18110180019@fudan.edu.cn)} ,
Simon Hubmer\footnote{Johann Radon Institute Linz, Altenbergerstra{\ss}e 69, A-4040 Linz, Austria, (simon.hubmer@ricam.oeaw.ac.at)} ,
Shuai Lu\footnote{Corresponding author. School of Mathematical Sciences, Fudan University, Shanghai 200433, China, (slu@fudan.edu.cn)} ,
Ronny Ramlau\footnote{Johannes Kepler University Linz, Institute of Industrial Mathematics, Altenbergerstra{\ss}e 69, A-4040 Linz, Austria, (ronny.ramlau@jku.at)} \footnote{Johann Radon Institute Linz, Altenbergerstra{\ss}e 69, A-4040 Linz, Austria, (ronny.ramlau@ricam.oeaw.ac.at)}
}

\begin{document}
\maketitle

\begin{abstract}
In this paper, we investigate regularization of linear inverse problems with irregular noise. In particular, we consider the case that the noise can be preprocessed by certain adjoint embedding operators. By introducing the consequent preprocessed problem, we provide convergence analysis for general regularization schemes under standard assumptions. Furthermore, for a special case of Tikhonov regularization in Computerized Tomography, we show that our approach leads to a novel (Fourier-based) filtered backprojection algorithm. Numerical examples with different parameter choice rules verify the efficiency of our proposed algorithm.
\end{abstract}

\noindent \textbf{Keywords.} Linear Inverse Problems, Embedding Operators, Irregular Noise, Regularization Theory, Computerized Tomography


\section{Introduction}\label{se1}

In this paper, we consider a general setting of linear inverse problem of the form
    \begin{equation}\label{eq_linearIP}
        \yd = A \fD + \delta \zeta \,,
    \end{equation}
where $A: X \to Y$ is a bounded linear operator acting between two Hilbert spaces $X$, $Y$, the element $\fD \in X$ is the unknown exact solution, and $\yd$ is the (noisy) measurement.
Furthermore, the non-negative constant $\delta$ in (\ref{eq_linearIP}) denotes the noise level, and $\zeta$ is the noise satisfying $\zeta\notin Y$, which is assumed to be preprocessed by certain embedding operators as specified below. These linear inverse problems arise frequently in imaging sciences including the imaging denoising, deblurring, deconvolution and several tomographic applications. 

Embedding operators play an important role in inverse problems when one aims to reconstruct some unknown function from (in)direct noisy measurement. As an illustration, consider the imaging denoising problem in the framework of \cite{Ramlau_Teschke_2004_1,Ramlau_Teschke_2004_2}, where the forward operator $A$ in \eqref{eq_linearIP} is chosen as the identity operator, and the unknown solution belongs to a Sobolev space, i.e.,  $\fD \in H^s(\R)$ for some $s > 0$. The noise $\zeta$ is assumed to be an element of $L^2(\R)$, and the Sobolev norm of an un-regularized reconstruction may deviate from that of the true one. To obtain an improved reconstruction, the Sobolev embedding operator $E_s: H^s(\R) \rightarrow L^2(\R)$ is introduced, and the original inverse problem is changed to
    \begin{equation*}
        f \in H^s(\R) \mapsto E_s f  + \delta \zeta = \fd \in L^2(\R)  \,.
    \end{equation*}
In order to stabilize the above denoising problem, the following Tikhonov regularization method incorporating the embedding operator $E_s$ was considered in \cite{Ramlau_Teschke_2004_1}
    \begin{equation*}
        \Ja(x) := \norm{\fd - E_s f }_{L^2(\R)}^2 + \alpha \norm{ f }_{H^s(\R)}^2 \,,
    \end{equation*}
where $\alpha$ is a regularization parameter. The minimizer $\fad$ of this functional is given by
    \begin{equation*}
        \fad = \kl{E_s^* E_s + \alpha I }^{-1} E_s^* \fd \,,
    \end{equation*}
from which explicit expression can be derived via Fourier analysis.

Besides Tikhonov regularization, many other regularization methods such as the truncated SVD method or Landweber iteration are available for solving linear inverse problems of the form \eqref{eq_linearIP}. Mathematically, these methods can be analysed in the following common framework \cite{Engl_Hanke_Neubauer_1996,Lu_Pereverzev_2013}. Let $\ga(\lambda)$ be a generic regularization function and $\ra(\lambda):= 1-\lambda \ga(\lambda)$ be a reconstruction error function. Then we define the regularized solution $\fad$ of the noisy measurement by
	\begin{equation*}
		\fad = \ga(A^*A) A^* \yd \,.
	\end{equation*}
The total error $\fD-\fad$ then can be expressed by
    \begin{equation*}
        \fD - \fad  = \ra(A^* A) \fD - \delta \ga(A^*A) A^* \zeta \,.
    \end{equation*}
Different regularization schemes are classified via varies forms of regularization and reconstruction error functions $\ga(\lambda)$ and $\ra(\lambda)$, which typically satisfy the following definition.

\begin{definition}\cite{Mathe_Pereverzev_2003}\label{defn_regularization}
A family of functions $\ga(\lambda)$ is called a regularization if it satisfies
    \begin{equation}\label{eq_regularization}
    \begin{split}
        & \sup_{0<\lambda\leq a } \abs{ 1-\lambda \ga(\lambda) } \leq \gamma \,,
        \qquad
        0 < \alpha \leq a \,,
        \\
        & \sup_{0<\lambda\leq a } \abs{ \ga(\lambda) } \leq \frac{\gamma_*}{\alpha} \,,
        \qquad
        0 < \alpha \leq a \,,
    \end{split}
    \end{equation}
with two finite constants $\gamma$, $\gamma_*$ and $a := \norm{A^*A}$. Moreover, the regularization function $\ga(\lambda)$ is said to have a qualification $\rho:(0,a)\rightarrow \R_+$ if the following inequality holds
    \begin{equation}\label{eq_regularization_qualification}
        \sup_{0<\lambda\leq a } \abs{ 1-\lambda \ga(\lambda)} \rho(\lambda) \leq \gamma \rho(\alpha) \,,
        \qquad
        0 < \alpha \leq a \,.
    \end{equation}
\end{definition}

\begin{example}\label{ex:regularization}
We present two classic regularization methods below:
\begin{itemize}
    \item \textbf{Tikhonov regularization} The regularization and reconstruction error functions for Tikhonov regularization are $\ga(\lambda):= \frac{1}{\lambda+\alpha}$ and $\ra(\lambda):= \frac{\alpha}{\lambda+\alpha}$, respectively, and the qualification is $\rho(\lambda):= \lambda$.
  \item \textbf{Truncated singular value decomposition (TSVD)} The regularization and reconstruction error functions for TSVD are
  \begin{align*}
  \ga(\lambda):=\left\{
  \begin{array}{cc}
    \lambda^{-1} \,, & \lambda \geq \alpha \,,  \\
    0 \,, & \lambda < \alpha \,,
  \end{array}
  \right. \quad {\textrm and}\quad
    \ra(\lambda):=\left\{
  \begin{array}{cc}
    0 \,, & \lambda \geq \alpha \,, \\
    1 \,, & \lambda < \alpha \,,
  \end{array}
  \right.
  \end{align*}
  respectively, and the qualification is $\rho(\lambda):= \lambda^n$ for arbitrary $n\in \mathbb{N}$.
\end{itemize}
For other regularization methods, these functions can be found e.g.\ in \cite{Engl_Hanke_Neubauer_1996,Lu_Pereverzev_2013}.
\end{example}

In this work, we are particularly interested in the case that the measurement noise does not belong to the space $Y$, but only to some more general space $Z$. For example, in a classical Sobolev setting one often assumes that the noise does not belong to $Y = \LtRn$ but to a negative order Sobolev space $Z = \HmsRn$ for some $s > 0$. In the following, we call this irregular noise. This is motivated by the observation that, e.g., white noise does not belong to the square integrable function space but to some negative Sobolev space almost surely \cite[Example 1]{Kekkonen_Lassas_Siltanen_2014}. More precisely, we will propose to preprocess the noisy measurements by using appropriate embedding operators, and discuss the solvability of the resulting inverse problem \eqref{eq_linearIP}. In order to reconstruct the unknown function $\fD$ from noisy measurements $\yd$ with irregular noise $\zeta$, a conventional approach is to assume that the variable $Af$ is sufficiently smooth for all $f \in X$ such that the inner product $\spr{ Af, \yd}$ is well-defined. In this case, one can replace the conventional residual $\norm{Af-\yd}_Y^2$ by $\norm{ Af }_Y^2 - 2\spr{ Af, \yd}_Y$ and reformulate the classic Tikhonov functional into
    \begin{equation*}
        \Ja(f) := \norm{Af}_Y^2- 2\spr{ Af, \yd}_Y + \alpha \norm{f}_X^2 \,.
    \end{equation*}
The above approach has been established and analysed for both Poissonian and Gaussian noise in \cite{Hohage_Werner_2013,Kekkonen_Lassas_Siltanen_2014}, respectively. Another approach to solve the inverse problem \eqref{eq_linearIP} with irregular noise is to preprocess the noisy measurement; see e.g.\ \cite{MR2855055,MR2772535,Blanchard_Mathe_2012,Lu_Mathe_2014}. For instance, by setting $H=A A^*$ and choosing an index $\mu >0$, one can assume that $H^{\mu}(\yd - Af) \in Y$ for $f \in X$ almost surely, which follows from Sazonov's theorem \cite{Sazonov_1958} if $H^{\mu}$ is a Hilbert-Schmidt operator. Here $H^{\mu}$ can be considered as an embedding or smoothing operator to weaken the influence of irregular noise. For example, choosing $\mu=1/2$ as in \cite{Lu_Mathe_2014} results in the original inverse problem \eqref{eq_linearIP} being replaced by a preprocessed (symmetrized) normal equation
    \begin{equation}\label{eq_linear_newIP}
        T f = g^{\delta} \,
    \end{equation}
by defining $T= A^* A$ and $g^{\delta} = A^* \yd$. Even though a straightforward implementation of different regularization schemes based on the symmetrized equation \eqref{eq_linear_newIP} is possible, when the noise is Gaussian, additional treatment concerning the choice of the regularization parameter is necessary in order to weaken the randomness of the noise. In particular, several modified discrepancy-based parameter choice rules with emergency stops, taking random noise with large deviation into account, have been proposed in \cite{Blanchard_Mathe_2012,Lu_Mathe_2014} to derive (order-optimal) error bounds when the unknown solution has some appropriate spectral resolution. For further treatments of inverse problems with random noise, we refer to the review paper \cite{Cavalier_2011} and the references therein.

At the same time, the combination of data preprocessing followed by an inversion algorithm has extensive applications. For example, a prominent reconstruction algorithms for computerized tomography (CT) is the {\it Filtered Backprojection,} which consists of a data preprocessing step in the Fourier space (usually some sort of frequency cutoff) followed by an application of the inverse operator, see e.g.~\cite{smith1985image,smith1985mathematical,shepp1974fourier,chang1980scientific}.
Similar approaches for general operator equations have have been investigated in \cite{Klann_Maass_Ramlau_2006,Klann_Ramlau_2008}. In this paper, we focus on the preprocessing approach discussed above, and consider regularization schemes for (\ref{eq_linear_newIP}) when the preprocessed noise is bounded in the Hilbert space $Y$. However, instead of preprocessing the noisy measurements by $H^{\mu} = (A^*A)^\mu$, we propose to use generic adjoint embedding operators associated with the forward operator. These operators can be well adapted to several applications, and we will particularly investigate their realization in CT with a novel filtered backprojection algorithm.

The outline of this paper is as follows: In Section~\ref{sect_background}, we summarize the definition and some properties of (negative) Sobolev spaces and their corresponding embedding operators. In Section~\ref{sect_error_bound}, we analyse the regularization properties of general regularization methods with preprocessing embedding operators. In Section~\ref{sect_Tikhonov}, we discuss an explicit application of our proposed approach in CT, and propose a Fourier-based Tikhonov regularization method to treat irregular noise, which yields a novel filtered backprojection algorithm. Finally, in Section~\ref{section_numerics}, we provide several numerical examples verifying the efficiency of our proposed algorithm in conjunction with different deterministic and heuristic parameter choice rules.

\section{Background on Sobolev embedding operators}\label{sect_background}

In this section, we review some background on real order Sobolev spaces and the (adjoint) Sobolev embedding operators, closely following \cite{Hubmer_Sherina_Ramlau_2023} and the works \cite{Adams_Fournier_2003,McLean_2000}. These Sobolev embedding operators can be good candidates to preprocessen the original forward problem (\ref{eq_linearIP}).

\subsection{Real order Sobolev spaces}\label{subsect_Sobolev}

First, we recall the definition of real order Sobolev spaces $\HsRn$ for arbitrary $s \in \R$; cf.~\cite{McLean_2000,Adams_Fournier_2003}. For this, let $\F : \SRn \to \SRn$ be the Fourier transform defined by
    \begin{equation}\label{def_F}
	   (\F u)(\xi) := \int_\Rn  u(x)  e^{-i x \cdot \xi} \, dx \,,
	   \qquad
	   \forall \, \xi \in \Rn \,,
    \end{equation}
where $\SRn$ denotes the Schwartz space of rapidly decreasing functions defined by
	\begin{equation*}
		\SRn = \Kl{ \phi \in C^\infty(\Rn) \,\, \vert \,\, \text{sup}_{x\in\Rn} \abs{ x^\alpha  \partial^\beta \phi(x)} < \infty \,,
        \forall \,\, \text{multi-indices} \, \,\alpha , \beta} \,.
	\end{equation*}
Note that the Fourier transform $\F$ and its inverse $\FI$ can be extended to
	\begin{equation*}
		\F : \SsRn \to \SsRn \,,
		\qquad
		\FI : \SsRn \to \SsRn \,,
	\end{equation*}
where $\SsRn$, the dual space of $\SRn$, is called the space of temperate distributions. Next, consider the continuous, linear Bessel potential operator $\Js : \SRn \to \SRn$ of order $s$ defined by
	\begin{equation*}
		\Js u(x) := (2\pi)^{-n}\int_\Rn (1 + \abs{\xi}^2)^{s/2} \F(u)(\xi) e^{i x \cdot \xi} \, d \xi \,,
		\qquad
		\forall \, x \in \Rn \,.
	\end{equation*}
The Bessel potential operator $\Js$ is selfadjoint w.r.t.\ the $\LtRn$ inner product, i.e.,
	\begin{equation*}
		\spr{\Js u , v}_\LtRn = \spr{u,\Js v}_\LtRn \,,
		\qquad
		\forall \, u ,v \in \SRn \,,
	\end{equation*}
which allows to naturally extend it to a linear operator $\Js : \SsRn \to \SsRn$. Furthermore, from its definition one sees that the Bessel potential operator satisfies
	\begin{equation*}
		\F\kl{ \Js u}(\xi) = (1 + \abs{\xi}^2)^{s/2} \F(u)(\xi) \,,
	\end{equation*}
from which it follows that its application amounts to a multiplication with $(1 + \abs{\xi}^2)^{s/2}$ in the Fourier domain, and thus $\Js$ can be seen as a kind of differential operator. With this, we can now define the fractional order Sobolev spaces $\HsRn$ of order $s \in \R$ by
	\begin{equation*}
		\HsRn := \Kl{ u \in \SsRn \, \vert \, \Js u \in \LtRn} \,,
	\end{equation*}
and equip this space with the inner product and induced norm
	\begin{equation}\label{def_norm_HsRN}
		\spr{u,v}_\HsRn := \spr{\Js u , \Js v}_\LtRn \,,
        \qquad
        \text{and}
        \qquad
        \norm{u}_\HsRn = \sqrt{\spr{u,u}_\HsRn}  \,,
	\end{equation}
respectively. Note that by Plancherel's theorem it follows that
	\begin{equation*}
		\norm{u}_\HsRn^2 = (2\pi)^{-2n}\int_\Rn (1 + \abs{\xi}^2)^{s} \abs{\F(u)(\xi)}^2 \, d \xi \,,
	\end{equation*}
and analogously for the inner product. Note that the definition of real-order Sobolev spaces $\HsO$ over non-empty, open sets $\Omega \subset \Rn$ is somewhat more involved; cf.~\cite{McLean_2000,Adams_1970}.

\subsection{Sobolev embedding operators}\label{subsect_embd}

Next, we consider embedding operators between the Sobolev spaces $\HsRn$ such as
    \begin{equation*}
        E_s : \HsRn \to \LtRn \,,
        \qquad
        u \mapsto E_s u := u \,.
    \end{equation*}
The properties of these embeddings and their generalization to bounded domains can e.g.\ be found in \cite{Adams_Fournier_2003}. For a collection of representations of their adjoint operators $E_s^*$ see \cite{Hubmer_Sherina_Ramlau_2023}. In this paper, we consider the related but slightly different embedding operators
	\begin{equation*}
		\Es : \LtRn \to \HmsRn \,, \quad u \mapsto \Es u := u \,.
	\end{equation*}
Since $\LtRn \subseteq \HmsRn$ for all $s \geq 0$ this operator is well-defined and bounded. Hence, it has a well-defined and bounded adjoint $\Es^* : \HmsRn \to \LtRn$, characterized by
	\begin{equation*}
		\spr{\Es^* u, v}_\LtRn  = \spr{u,v}_\HmsRn \,,
		\qquad
		\forall \, u \in \HmsRn \,, v \in \LtRn \,.
	\end{equation*}
From this characterization, we find with the definition \eqref{def_norm_HsRN} of the inner product that
	\begin{equation}\label{dual_E}
		\Es^* u = \FI\kl{ (1 + \abs{\cdot}^2)^{-s} \F(u)(\cdot) } = \mathcal{J}^{-2s} u \,,
		\qquad
		\forall \, u \in \HmsRn \,.
	\end{equation}
For further considerations, note that by the definition of the norm on $\HsRn$ we have
	\begin{equation*}
		\norm{u}_\HmsRn = \norm{\Jms u}_\LtRn  = \norm{\Esh^* u}_\LtRn \,.
	\end{equation*}


\section{Error bound analysis for general regularization schemes}\label{sect_error_bound}

In this section, we recall the generic linear inverse problem \eqref{eq_linearIP}, i.e.,
    \begin{equation*}
        \yd = A \fD + \delta \zeta \,.
    \end{equation*}
As noted above, we are interested in the case that the noise is irregular, i.e., $\zeta \notin Y$ but $\zeta \in Z$ for some larger Hilbert space $Z \supset Y$. This may for example be a negative order Sobolev spaces as chosen in the next section, but the subsequent analysis is not restricted to this specific choice. Following our discussion in Section \ref{se1}, we assume that the irregular noise $\zeta$ can be preprocessed and the following assumption holds.
\begin{assumption}\label{ass_noise}
There exists a bounded linear embedding operator $\E : Y \to Z, y \mapsto y$ such that $\eta := \E^* \zeta \in Y$ satisfies $\norm{\eta}_Y \leq 1$. 
\end{assumption}

Using the adjoint embedding operator $\E^*$, we now define
    \begin{equation*}
        T := \E^* A \,,
        \qquad
        g := \E^* y \,,
        \qquad
        \gd := \E^* \yd \,,
    \end{equation*}
and, instead of the original inverse problem \eqref{eq_linearIP}, consider the preprocessed inverse problem \eqref{eq_linear_newIP}, i.e.,
    \begin{equation*}
        T f = \gd = g + \delta\E^*\zeta \,.
    \end{equation*}

\begin{example}\label{ex:ct}
A prototypical example for (\ref{eq_linear_newIP}) is given by the Radon transform \cite{Natterer_1985}; see in particular \cite{Mathe_2019}. Let $\Omega\subset \Rn$ be the unit ball in $\Rn$ and let $f \in \LtO$. Then for $\theta \in \mathbb{S}^{n-1}$ and $\kappa \in \R$ the Radon transform is defined by
    \begin{align*}
		  Rf(\theta,\kappa) :=\int_{x\cdot\theta=\kappa}f(x)dx=\int_{\theta^{\bot}}f(\kappa\theta+y)dy \,,
	\end{align*}
and the inverse Radon problem consists of solving the operator equation $y = Rf$. The Radon transform is typically considered as a linear operator $R: \LtO \to \LtSnmR$. Furthermore, in practice one is normally faced with noisy data of the form
	\begin{align*}
		y^{\delta} = Rf + \delta \zeta.
	\end{align*}
In some cases, the noise $\zeta$ may not belong to $\LtSnmR$ but only to some negative order Sobolev space. Hence, it follows that $y^{\delta} \notin \LtSnmR$, and thus the problem has to be preprocessed for example by Sobolev embedding operators as discussed above. We will revisit this particular example in detail in Section~\ref{sect_Tikhonov}.
\end{example}

To further regularize the above preprocessed forward problem, we can define the regularized solutions
    \begin{equation}\label{eq_generalminimizer}
        \fad := \ga(T^*T) T^* \gd \,,
    \end{equation}
where $\ga$ is a general regularization function; cf.~Definition~\ref{defn_regularization}. The aim of this section is to establish general error bounds for the regularized solutions \eqref{eq_generalminimizer}. For this, note first that one can decompose
    \begin{equation}\label{eq_mainerror}
        \fD-\fad = \kl{ I-\ga(T^*T)T^*T } \fD - \ga(T^*T)T^*\delta \eta \,.
    \end{equation}
Hence, in order to estimate $\norm{ \fD-\fad}_{X}$ we now establish bounds for both terms in the above equality. To this end, we need to impose certain smoothness assumptions on the unknown solution $\fD$ in the form of \emph{general source conditions} c.f. \cite{Lu_Pereverzev_2013,Mathe_Pereverzev_2003,Mathe_2019}, for which we make the following definition.

\begin{definition}\label{defn_indexfun}\cite{Lu_Pereverzev_2013,Mathe_Pereverzev_2003,Mathe_2019}
The function $\varphi:[0,\infty)\to[0,\infty)$ is called an \emph{index function} if it is non-decreasing, continuous, positive on $(0,\infty)$, and satisfies $\varphi(0)=0$. Given two index functions $\varphi_1$, $\varphi_2$, we define the partial ordering $\varphi_1 \prec \varphi_2$ (read $\varphi_1$ is beyond $\varphi_2$) if the function $t\rightarrow \varphi_1(t)/\varphi_2(t)$ is an index function, i.e., $\varphi_1$ tends to zero faster than $\varphi_2$.
\end{definition}

Next, we introduce the following standard smoothness assumption.

\begin{assumption}\label{ass_source-set}
 There exists an index function $\varphi$ such that there holds
    \begin{equation}\label{eq:source-set}
        \fD \in \mathcal{A}_{\varphi} = \Kl{f \in X \,\vert\, \exists \,  w \in X \,, \norm{w}_X \leq 1 : \quad f = \varphi(A^*A) w \,  } \,.
    \end{equation}
\end{assumption}

Furthermore, we require the following link condition between $A A^*$ and $\E \E^*$:

\begin{assumption}\label{ass_linkcond}
There exist an index function $\psi$ and constants $0< m \leq 1 \leq M < \infty$ such that
    \begin{equation*}
        m \norm{\psi(A A^* ) v }_Y \leq  \norm{ (\E \E^* )^{1/2} v}_Y \leq M \norm{\psi(A A^* ) v }_Y\,,
        \qquad
        \forall \, v \in Y \,.
    \end{equation*}
\end{assumption}

\begin{example}\label{ex_linkcondition}
We give an example of the above assumption. Let $A$ be self-adjoint and assume that the operators $\E \E^*$ and $A A^*$ obey
	\begin{itemize}
	    \item for some $a>0$, we have the eigenvalues $s_j(\E \E^*) \asymp j^{-(1+2a)}$ for $j=1,2,\dots$,
    	\item for some $p>0$, we either have the eigenvalues $s_j(A A^*) \asymp j^{-2p}$, or $s_j(A) \asymp j^{-p}$, for $j=1,2,\ldots$ .
	\end{itemize}
In this setting, we obtain the index function $\psi(t) = t^{\frac{1+2a}{4p}}$ in Assumption \ref{ass_linkcond}.
\end{example}

Assumption~\ref{ass_linkcond} can also be used to establish a further linking condition between $A^*A$ and the modified forward operator $T^* T$, as we see in

\begin{lemma}
Let Assumption~\ref{ass_linkcond} hold and denote $\Theta(\lambda) := \sqrt{\lambda}\psi(\lambda)$. Then there holds
    \begin{equation}\label{eq_linkcond}
        m\norm{\Theta(A^* A) w }_X \leq \norm{ (T^* T)^{1/2} w}_X \leq M\norm{\Theta(A^* A) w }_X\,,
        \qquad
        \forall \, w\in X \,.
    \end{equation}
\end{lemma}
\begin{proof}
Let $w \in X$ and denote $v := Aw$. Then due to Assumption~\ref{ass_linkcond} there exists an index function $\psi$ and constants $0< m \leq 1 \leq M < \infty$ such that
    \begin{equation*}
        m \norm{\psi(A A^* ) v }_Y \leq  \norm{ (\E \E^* )^{1/2} v}_Y \leq M \norm{\psi(A A^* ) v }_Y\,,
    \end{equation*}
which we can square to obtain the inequalities
    \begin{equation*}
    \begin{split}
        m^2 \spr{ \psi(A A^* ) A w, \psi(A A^* ) A w }_Y
        &\leq
        \spr{ (\E \E^*)^{1/2} A w , (\E \E^*)^{1/2} A w}_Y
        \\
        &\leq M^2 \spr{ \psi(A A^* ) A w, \psi(A A^* ) A w }_Y \,.
    \end{split}
    \end{equation*}
Since both $\psi(A A^*)$ and $(\E \E^*)$ are selfadjoint, we can rewrite this into
    \begin{equation*}
        m^2 \spr{ \psi^2(A^* A ) A^* A w, w }_X
        \leq
        \spr{ A^* \E \E^* A w ,  w}_X
        \leq
        M^2 \spr{ \psi^2(A^* A ) A^* A w, w}_X \,.
    \end{equation*}
Finally, by definition we have $T := \E^* A$ it follows that
    \begin{equation*}
        m^2 \spr{ \psi^2(A^* A ) A^* A w, w }_X
        \leq
        \spr{ T^*T w, w }_X
        \leq
        M^2 \spr{ \psi^2(A^* A ) A^* A w, w}_X \,,
    \end{equation*}
which together with the definition of $\Theta$ yields the assertion.
\end{proof}

From the properties of index functions it follows that $\Theta^2(\lambda)$ is strictly monotone and increases superlinearly. Next, we follow \cite{Mathe_2019} and define the related function
    \begin{equation}\label{eq_relatedfun}
        f_0(t) := \kl{(\Theta^2)^{-1} (t)}^{1/2} \,,
        \qquad
        \forall \, t > 0 \,,
    \end{equation}
and assume that it is such that $f_0^2$ is an operator concave function. Furthermore, in order to take into account the qualification of the regularization method, cf.~Definition~\ref{defn_regularization}, we need a stronger lifting condition than the one in \eqref{eq_linkcond}, and thus make

\begin{assumption}[\cite{Mathe_2019}, Lifting condition]\label{ass_liftingcond}
There exists a lifting index $\mu>1$ and some constants $0< m \leq 1 \leq M < \infty$ such that
    \begin{equation*}
        m^{\mu}\norm{\Theta^{\mu}(A^* A) w }_X \leq \norm{ (T^* T)^{\mu/2} w}_X \leq M^{\mu} \norm{\Theta^{\mu}(A^* A) w }_X\,,
        \qquad
        \forall \, w \in X \,.
    \end{equation*}
Furthermore, the function $f_0^2$ with $f_0$ as in \eqref{eq_relatedfun} is an operator concave function.
\end{assumption}

Next, we present our main error bounds for general regularization schemes \eqref{eq_generalminimizer} following the arguments in \cite{Mathe_2019,MR2772535}.

\begin{theorem}\label{thm_generalerror}
Let $\varphi$ be an index function satisfying Definition \ref{defn_indexfun}, $\ga(\lambda)$ be a regularization function satisfying Definition~\ref{defn_regularization} with qualification $\rho$ as in \eqref{eq_regularization_qualification} and Assumptions \ref{ass_noise}-\ref{ass_linkcond} hold.
\begin{enumerate}
    \item If $1 \prec \varphi \prec \Theta$ and the function $\lambda \mapsto \varphi(f_0^2(\lambda))/\rho(\lambda)$ is non-increasing, then the error between the unknown true solution $\fD$ and the regularization minimizer $\fad$ in \eqref{eq_generalminimizer} can be bounded by
        \begin{equation*}
            \norm{\fD-\fad}_{X}\leq C \kl{\varphi(f_0^2 (\alpha)) + \frac{\delta}{\sqrt{\alpha}}} \,,
        \end{equation*}
where the constant $C$ depends on the regularization schemes and the linking condition.
    \item If $\Theta \prec \varphi \prec \Theta^\mu$, the function $\lambda \mapsto \varphi(f_0^2(\lambda))/\rho(\lambda)$ is non-increasing, and Assumption~\ref{ass_liftingcond} holds with a sufficiently large $\mu$, then the error between the unknown true solution $\fD$ and the regularization minimizer $\fad$ in \eqref{eq_generalminimizer} can be bounded by
        \begin{equation*}
            \norm{\fD-\fad}_{X} \leq C \kl{\varphi(f_0^2 (\alpha)) + \frac{\delta}{\sqrt{\alpha}}} \,,
        \end{equation*}
    where the constant $C$ depends on the regularization schemes and the lifting condition.
\end{enumerate}
\end{theorem}
\begin{proof}
First of all, by recalling \eqref{eq_mainerror} we bound the total error by
    \begin{equation*}
    \begin{split}
        \norm{\fD-\fad}_{X} & \leq \norm{\fD-\fa} + \norm{\fa -\fad } \\
        & \leq \norm{\kl{ \ga(T^* T) T^* T - I} \fD } + \norm{\ga(T^* T) T^* \delta \eta} \,.
    \end{split}
    \end{equation*}
For the first term on the right side we use the source condition $\fD: = \varphi(A^*A) w$ to obtain
    \begin{equation}\label{eq_bias}
        \norm{\kl{ \ga(T^* T) T^* T - I} \fD }  = \norm{ \ra (T^* T) \varphi(A^*A) w } .
    \end{equation}
Next, we derive bounds for this term depending on the following two cases:

\begin{enumerate}
    \item In the first case, we assume that $1 \prec \varphi \prec \Theta$. Then by noticing $\varphi^2(f_0^2 (t))$ is operator concave, the interpolation theorem from \cite{Mathe_Tautenhahn_2006} yields
        \begin{equation*}
            \norm{\varphi(A^*A) w} \leq M \norm{\varphi(f_0^2 (T^*T)) w}, \qquad  \forall \, w\in X \,.
        \end{equation*}
    Thus, the Douglas' range inclusion theorem \cite{Douglas_1996,Mathe_2019} yields that for any $w$, $\norm{w}\leq 1$ we can find $\hat{w}$, $\norm{\hat{w}}\leq M$ such that $\varphi(A^*A) w = \varphi(f_0^2(T^*T)) \hat{w}$, which yields
        \begin{equation*}
        \begin{split}
            \norm{\kl{ \ga(T^* T) T^* T - I} \fD }^2 & \leq \norm{ \ra (T^* T) \varphi(A^*A) w }^2
            \quad
            \\
            & =\norm{ \ra (T^* T) \varphi(f_0^2(T^*T)) \hat{w} }^2 \\
            & \leq M^2  \norm{ \ra (T^* T) \varphi(f_0^2(T^*T))}^2 \leq C \varphi^2(f_0^2(\alpha)) \,,
        \end{split}
        \end{equation*}
    since the function $\lambda \mapsto \varphi(f_0^2(\lambda))/\rho(\lambda)$ is non-increasing
    \item In the second case, we assume that Assumption \ref{ass_liftingcond} holds with a sufficiently large $\mu$ such that $\varphi^2((\Theta^{2\mu})^{-1}(t))$ is operator concave. Thus, together with the interpolation theorem we obtain
        \begin{equation*}
            \norm{\varphi(A^*A)w} \leq M \norm{\varphi\kl{\kl{\Theta^{2\mu}}^{-1} ((T^*T)^{\mu})}} \,.
        \end{equation*}
    Referring to the proof of \cite[Prop.5]{Mathe_2019}, we have $\kl{\Theta^{2\mu}}^{-1} (\lambda) = \kl{\Theta^{2}}^{-1} (\lambda^{1/\mu})$, which can be verified by applying $\Theta^{2\mu}$ on both sides. Thus, we obtain the bound
        \begin{equation*}
        \begin{split}
            \norm{\kl{ \ga(T^* T) T^* T - I} \fD }^2 & \leq C \norm{ \ra (T^* T) \varphi \kl{\kl{\Theta^{2\mu}}^{-1} ((T^*T)^{\mu})} \hat{w}  }^2
            \\
            & = C \norm{ \ra (T^* T)\varphi \kl{\kl{\Theta^{2}}^{-1} (T^*T)} \hat{w}  }^2
            \\
            & \leq  C \varphi^2(f_0^2(\alpha)) \,,
        \end{split}
        \end{equation*}
    where we used a similar calculation as in the previous case.
\end{enumerate}
Next, we consider the second term in \eqref{eq_mainerror}, i.e., $\ga(T^* T) T^* \delta \eta$. By Definition \ref{defn_regularization} and Assumption \ref{ass_noise}, we bound this term by classic techniques such that $\|\ga(T^* T) T^* \delta \eta\| \leq C \frac{\delta}{\sqrt{\alpha}}$. Finally, we combine the bounds for both terms, which yields the assertion.
\end{proof}

\begin{remark}
The first part of Theorem~\ref{thm_generalerror} is a simplified version of \cite[Theorem 4]{MR2772535} and both parts are consistent with those in regularization theory and Bayesian inference, see e.g.~\cite{Blanchard_Mathe_2012,Lu_Mathe_2014,Lin_Lu_Mathe_2015,Mathe_2019}. For illustration, we recall Example \ref{ex_linkcondition} and additionally assume the source condition such that
	\begin{itemize}
    	\item for some $\beta>0$, the exact solution $\fD$ has a coefficient expansion $\{\fD_j\}$ with respect to the eigensystem of $A$, and satisfies $\sum_{j=1}^{\infty} j^{2\beta} (\fD_j)^2 \leq 1$.
	\end{itemize}
Then the index function in Assumption \ref{ass_linkcond} is $\psi(t) = t^{\frac{1+2a}{4p}}$ and $\Theta^2(t) = t^{\frac{1+2a+2p}{2p}}$ in \eqref{eq_linkcond}. Furthermore, for the source condition we have $\varphi(t) = t^{\frac{\beta}{2p}}$. By Theorem~\ref{thm_generalerror}, for Tikhonov regularization, we obtain the total error bound
    \begin{equation*}
        \norm{\fD-\fad}_X \leq \alpha^{\frac{\beta}{1+2a+2p}} + \frac{\delta}{\sqrt{\alpha}} \,,
    \end{equation*}
with $\beta \in (0, 1+2a+2p]$ referring to the qualification in Assumption \ref{ass_source-set}. For other regularization schemes, one can obtain the same rate but with (in general) a larger interval of $\beta$ induced by their corresponding qualification.
\end{remark}

\section{Tikhonov regularization for computerized tomography via Sobolev embedding operators}\label{sect_Tikhonov}

In this section, we consider a Fourier-based Tikhonov regularization method for computerized tomography by implememting Sobolev embedding operators.

\subsection{Brief overview of the Radon transform}

We first provide an extended overview of the Radon transform referring to Example \ref{ex:ct}, which is of particular importance in CT. There, it is involved in the description of the following simple physical model of X-ray attenuation: Let the density function $f(x)$, compactly supported on the unit ball $\Omega \subset \Rn$, denote the X-ray attenuation coefficient of some tissue at the point $x$ and let $I_0$ be the intensity of the X-ray sent from its source. Then following the Lambert-Beer law, the intensity received by the detector behind the object along the straight line $\ell$ is given by
    \begin{equation*}
        I = I_0 \exp\kl{ \int_\ell f(x) \, dx } \,.
    \end{equation*}
These intensities are measured for all lines $\ell$ passing through the scanned object, with the lines themselves being characterized by an angular $\theta \in \mathbb{S}^{n-1}$ and a radial component $\kappa \in \R$, respectively. Defining the Radon transform as in Example~\ref{ex:ct}, i.e.,
    \begin{equation}\label{eq_RadonTransform}
    \begin{split}
        R: \LtO &\to L^2(\mathbb{S}^{n-1}\times\R)
        \\
        f(x)&\mapsto  Rf(\theta,\kappa):=\int_{x\cdot\theta=\kappa}f(x)dx=\int_{\theta^{\bot}}f(\kappa\theta+y)dy \,,
    \end{split}
    \end{equation}
the CT problem of estimating the density $f$ can thus be written in the form \eqref{eq_linearIP}, i.e.,
    \begin{equation}\label{o1}
        \yd:= \yd(\theta,\kappa)= (Rf)(\theta, \kappa) +\delta\zeta(\theta,\kappa) \,,
    \end{equation}
where $\zeta$ is the measurement noise and $\delta \geq 0$ is the noise level.

There are numerous algorithms for the inversion of the Radon transform, c.f.~\cite{Natterer_2001}, such as the classic inverse Radon transform, the filtered backprojection (FBP), or Fourier-based reconstruction methods. In case that there is no noise in the measurement, these algorithms allow to reconstruct the density function $f$ with high resolution. In the practically more relevant case of uniformly bounded measurement noise, regularization schemes have been used for stable reconstruction \cite{Natterer_2001}. Some classical choices are the algebraic reconstruction technique (ART, e.g., Kaczmarz's method and its variants \cite{Natterer_2001}), or the simultaneous algebraic reconstruction technique (SART, e.g., Landweber iteration and its variants \cite{Guan_Gordon_1996,Xu_Liow_Strother_1993,Andersen_Kak_1984}), although some other regularization methods for linear inverse problems can also be used \cite{Engl_Hanke_Neubauer_1996}. However, little is known specifically when the Radon transform is considered with irregular noise.

In this section, we focus on the Fourier-based reconstruction algorithms for the Radon transform. An asymptotic analysis of this type of reconstruction algorithms was given in \cite{Natterer_1985} when the measurements contain uniformly bounded noise. In particular, a modified inversion algorithm was also proposed in which the Fourier data on the Cartesian grid was obtained by moving the points onto the closest straight line on a polar grid, showing that the modified inversion algorithm is asymptotically optimal. In \cite{Cheung_Lewitt_1991}, a modified polar grid was proposed, which made the sampling data more efficient and the numerical reconstruction more accurate. Furthermore, a gridding method was proposed in \cite{Schomberg_Timmer_1995} which involves a window function to smoothen the convolution and provides a fast, accurate alternative algorithm to the FBP method. Later, a non-equispaced fast Fourier transform algorithm based on an exact Fourier series representation was developed in \cite{Fourmont_2003}. It is also worth mentioning that in \cite{Reynolds_Matthew_Beylkin_Monzan_2013} a new and fast polar coordinate Fourier domain algorithm was considered, which uses optimal rational approximations of projection data collected in X-ray tomography.

Despite the fact that Fourier-based reconstruction algorithms become more accurate, the effect of irregular noise is not particularly investigated. In order to deal with such noise, we now follow our approach above and derive a novel regularization scheme for the Radon transform. This scheme makes use of the Fourier slice theorem, which reconstructs the unknown density function $f$ by converting a collection of projection data into two-dimensional Fourier data on a polar grid.

\subsection{Fourier-based inversion algorithm for the Radon transform}

In this subsection, we consider the realization of the Fourier-based Tikhonov regularization method for CT by implementing Sobolev embedding operators. In particular, we derive an explicit form of the minimizer of the corresponding functional which serves as the basis for the numerical computation in the subsequent section. To this end, we first recall the definition and some properties of the Fourier transform in relation to the Radon transform which can be found, e.g., in \cite{Natterer_2001}. Please note that these results can be extended to compactly supported functions in $\LtRn$, in particular to $f \in \LtO$.

\begin{definition}\cite{Natterer_2001}\label{Def_1}
Let $f \in \S(\R^n)$, the Fourier transform $\F f$ is defined by \eqref{def_F}, and the inverse Fourier transform $\F^{-1} f$ is defined by
    \begin{equation*}
        \F^{-1} f (x) := \frac{1}{(2\pi)^n}\int_{\R^n} f(\xi)e^{ix\cdot\xi} \, d\xi \,,
        \qquad
        \forall \, x\in\R^n \,.
    \end{equation*}
\end{definition}

Next, for all $h\in\S(\mathbb{S}^{n-1}\times\R)$ we define $\F _{2}(h)(\theta,\sigma) :=\int_{\R}h(\theta,\kappa)e^{-i\kappa\sigma}d\kappa$ as the Fourier transform with respect to the second variable $\sigma$. Analogously, $\F ^{-1}_{2}(h)(\theta,\kappa):=\frac{1}{2\pi}\int_{\R}h(\theta,\sigma)e^{i\kappa\sigma}d\sigma$ denotes the inverse Fourier transform with respect to this variable.

\begin{lemma}\cite[Prop.~11 and 15]{Natterer_2001}
For $f\in\S(\R^n)$ and $h\in\S(\mathbb{S}^{n-1}\times\R)$, there holds
    \begin{equation}\label{eq_RF}
        \F _{2}(Rf)(\theta,\sigma)=\F (f)(\sigma\theta) \,,
    \end{equation}
and
    \begin{equation}\label{eq_RsF}
        \F (R^*h)(\xi)=(2\pi)^{n-1}\abs{\xi}^{1-n}\kl{\F _{2}(h)\kl{\frac{\xi}{\abs{\xi}},\abs{\xi}}+\F _{2}(h)\kl{-\frac{\xi}{\abs{\xi}},-\abs{\xi}}} \,.
    \end{equation}
\end{lemma}

Next, in order to fit the Radon transform we consider the embedding operator
    \begin{equation*}
        \Is : L^2(\mathbb{S}^{n-1}\times\R)\to H^{-s}(\mathbb{S}^{n-1}\times\R) \,,
    \end{equation*}
where analogously to Section~\ref{sect_background} the spaces $H^{s}(\mathbb{S}^{n-1}\times\R)$ are defined by
	\begin{equation*}
		H^{s}(\mathbb{S}^{n-1}\times\R) := \Kl{ u \in \S^*(\mathbb{S}^{n-1}\times\R) \, \vert \, \mathcal{F}_2^{-1}\kl{ (1 + \abs{\sigma}^2)^{s/2} \mathcal{F}_2(u)(\theta,\sigma)} \in L^2(\mathbb{S}^{n-1}\times\R)} \,,
	\end{equation*}
and are equipped with the inner product
	\begin{equation}\label{eq_helper}
		\spr{u,v}_{H^{s}(\mathbb{S}^{n-1}\times\R)} := (2\pi)^{-2} \int_{\mathbb{S}^{n-1}} \int_\R  (1 + \abs{\sigma}^2)^{s} \mathcal{F}_2(u)(\theta,\sigma)\mathcal{F}_2(v)(\theta,\sigma) \, d\sigma d\theta  \,.
	\end{equation}
As in Section~\ref{subsect_embd} we have that the operator $\Is$ is well-defined and, due to \eqref{eq_helper}, its adjoint operator $\Is^* : H^{-s}(\mathbb{S}^{n-1}\times \R)\to L^2(\mathbb{S}^{n-1}\times\R)$ can be characterized by
    \begin{equation}\label{dual_E_Radon}
    	\Is^* u = \mathcal{F}_2^{-1}\big( (1 + \abs{\sigma}^2)^{-s} \mathcal{F}_2(u)(\theta,\sigma) \big).
    \end{equation}
As above, we assume that the measurement is corrupted by irregular noise, i.e., $\yd \in H^{-s}(\mathbb{S}^{n-1} \times \R)$ for some $s > 0$, and define the preprocessed data $\gd :=\Is^* \yd$ and the modified forward operator $\Rt := \Is^* R \, : \LtO \to L^2(\mathbb{S}^{n-1}\times\R)$. With this we, analogously to \eqref{eq_linear_newIP}, obtain the preprocessed problem
    \begin{equation}\label{eq_preRadon}
        \gd = \Rt f + \delta \eta \,,
        \qquad
        \text{where}
        \qquad
        \eta = \Is^* {\zeta}
    \end{equation}
which yields the following Tikhonov regularization method
    \begin{equation}\label{re}
        \fad: = \underset{f \in \LtO}{\arg\min} \Kl{\norm{\Rt f-\gd}_{\LtSnmR}^2 + \alpha \norm{f}_{\LtO}^2} \,.
    \end{equation}
The above minimizer $\fad$ can be calculated explicitly and is derived below.

\begin{theorem}\label{t1}
The minimizer $\fad$ of the Tikhonov functional \eqref{re} is characterized by
    \begin{equation*}
        \kl{\F \fad}(\xi) = \frac{\F _2{y}^{\delta}\kl{\frac{\xi}{\abs{\xi}},\abs{\xi}} +\F _2{y}^{\delta}(-\frac{\xi}{\abs{\xi}},-\abs{\xi})} {2+\alpha(2\pi)^{1-n}\abs{\xi}^{n-1}(1+\abs{\xi}^2)^s} \,, \quad \xi\in\Rt \,.
    \end{equation*}
In particular, if $\alpha=0$ the above formula is equivalent to the Fourier slice theorem.
\end{theorem}
\begin{proof}
First of all, we note that the minimizer $\fad$ of \eqref{re} is given as the solution of
    \begin{equation*}
        (\Rt^*\Rt+\alpha I)\fad = \Rt^* \gd \,.
    \end{equation*}
Taking the Fourier transform on both sides yields
    \begin{equation}\label{eq_Radon_minimizer}
        \F(\Rt^*\Rt \fad) + \alpha \F \fad = \F(\Rt^* \gd) \,.
    \end{equation}
Next, note that since $\Rt = \Is^* R$ and $\gd = \Is^*\yd$, it follows with \eqref{eq_RsF} that
    \begin{equation*}
    \begin{split}
       \F (\Rt^*\gd)(\xi)
        &=
        \F (R^*\Is \Is^* \yd)(\xi)
        =
        \F (R^*\Is^* \yd)(\xi)
        \\
        & \overset{\eqref{eq_RsF}}{=}
        (2\pi)^{n-1}\abs{\xi}^{1-n}\kl{\F_2(\Is^* \yd) \kl{\frac{\xi}{\abs{\xi}},\abs{\xi}} +\F_2(\Is^* \yd)\kl{-\frac{\xi}{\abs{\xi}},-\abs{\xi}}} \,,
    \end{split}
    \end{equation*}
and thus together with the Fourier characterization \eqref{dual_E_Radon} of $\Is^*$ we obtain
   \begin{equation}\label{A5}
    \begin{split}
        \F (\Rt^*\gd)(\xi)
        =
        (2\pi)^{n-1}\abs{\xi}^{1-n}(1+\abs{\xi}^2)^{-s}\kl{\F _2(\yd)\kl{\frac{\xi}{\abs{\xi}},\abs{\xi}}+\F _2(\yd)\kl{-\frac{\xi}{\abs{\xi}},-\abs{\xi}}} \,.
    \end{split}
    \end{equation}
Similarly, for the term $\F (\Rt^*\Rt f)(\xi)$ we obtain together with \eqref{eq_RF} and \eqref{eq_RsF} that
    \begin{equation*}
    \begin{split}
        &\F (\Rt^*\Rt \fad)(\xi)
        =
        \F (R^*\Is  \Is^*  R \fad )(\xi)
        \\
        &\qquad \overset{\eqref{eq_RsF}}{=}
        (2\pi)^{n-1}\abs{\xi}^{1-n}\kl{\F_2( \Is  \Is^*  R \fad)\kl{\frac{\xi}{\abs{\xi}},\abs{\xi}} +\F_2( \Is  \Is^*  R \fad)\kl{-\frac{\xi}{\abs{\xi}},-\abs{\xi}}}
        \\
        &\qquad \overset{\eqref{dual_E_Radon}}{=}
        (2\pi)^{n-1}\abs{\xi}^{1-n}(1+\abs{\xi}^2)^{-s}\kl{\F_2(R\fad)\kl{\frac{\xi}{\abs{\xi}},\abs{\xi}} +\F_2(R\fad)\kl{-\frac{\xi}{\abs{\xi}},-\abs{\xi}}}
        \\
        &\qquad \overset{\eqref{eq_RF}}{=}
        2(2\pi)^{n-1}\abs{\xi}^{1-n}(1+\abs{\xi}^2)^{-s}\F (\fad)(\xi) \,.
    \end{split}
    \end{equation*}
Inserting this and \eqref{A5} into \eqref{eq_Radon_minimizer} we thus obtain that
    \begin{equation*}
    \begin{split}
        & 2(2\pi)^{n-1}\abs{\xi}^{1-n}(1+\abs{\xi}^2)^{-s}\F (\fad)(\xi)
        +
        \alpha \F (\fad)(\xi)
        \\
        & \qquad =
        (2\pi)^{n-1}\abs{\xi}^{1-n}(1+\abs{\xi}^2)^{-s}\kl{\F _2(\yd)\kl{\frac{\xi}{\abs{\xi}},\abs{\xi}}+\F _2(\yd)\kl{-\frac{\xi}{\abs{\xi}},-\abs{\xi}}} \,,
    \end{split}
    \end{equation*}
which yields the assertion after rearranging the terms.
\end{proof}

Concerning the convergence analysis of our Fourier-based Tikhonov regularization scheme \eqref{re} we refer to Theorem~\ref{thm_generalerror}. 

\section{Numerical experiments}\label{section_numerics}

In this section, we describe the numerical realization of the Fourier-based Tikhonov regularization method \eqref{re} and compare its performance with some other approaches.

\subsection{Numerical reconstruction algorithm}

Since we assume that the density function $f$ is compactly supported, it can be represented by an $2M\times 2M$ pixel image which is zero when $\abs{x} > \varrho$ for some $\varrho > 0$. Naturally, $f$ cannot be band-limited, but we may assume that its Fourier transform is concentrated within $[-N,N]^2$. Furthermore, we assume that the original image $f(x)$ is zero outside $[-\tau,\tau]\times[-\tau,\tau]$ for some $\sqrt{\tau} \leq \varrho$, and define the mesh size $h_x=\tau/M$. Based on this discretization, we now derive a numerical method for calculating the minimizer $\fad$ of our reconstruction approach \eqref{re} for exact measurements $y$, as well as the slight modification necessary to also treat noisy measurements $\yd$.

First, recall that due to Theorem~\ref{t1} the minimizer of \eqref{re} is characterized by
    \begin{equation}\label{A2}
        \F \fa (\xi)
        =
        \frac{\F_2{y}^{\delta}\kl{\frac{\xi}{\abs{\xi}},\abs{\xi}}
        + \F_2 {y}^{\delta}\kl{ -\frac{\xi}{\abs{\xi}},-\abs{\xi}}} {2+\alpha(2\pi)^{1-n}\abs{\xi}^{n-1}(1+\abs{\xi}^2)^s} \,.
    \end{equation}
Since the computation of $\fa$ via this formula is not straightforward, we now describe the detailed numerical realization of \eqref{A2}. We start by estimating the value of $\F_2 y^{\delta}(\frac{\xi}{\abs{\xi}},\abs{\xi})$ by giving the parallel beam discrete observation in the half cylinder. For this, suppose that $y^{\delta}(\theta,\kappa)$ is available for the following values of $\theta$ and $\kappa$:
    \begin{equation*}
    \begin{split}
        \theta &= \theta_j
        = \Kl{\begin{matrix}\cos\phi_j\\ \sin\phi_j\end{matrix} }\,,
        \quad \phi_j =\frac{j\pi}{p} \,,
        \quad j=0\,, \dots \,, p-1 \,,
        \\
        \kappa &= \kappa_l = \frac{l\varrho}{q} \,,
        \quad l = -q \,, \dots \,, q \,.
    \end{split}
    \end{equation*}
We note that the measurement $y^{\delta}$ is zero for $\abs{\kappa}>\varrho$, since we have assumed that the density function $f$ is compactly supported. Next, we consider the evaluation of
    \begin{equation*}
        \F _2 y^{\delta}\kl{\frac{\xi_k}{\abs{\xi_k}},\abs{\xi_k}} \,,
        \quad \xi_k = (k_1 h_{\xi},k_2 h_{\xi}) \,,
        k=(k_1,k_2) \in \Z^2\,,
        k_I = -dM\,,\dots\,,dM (I=1,2) \,,
    \end{equation*}
where $d\geq1$ is the oversampling rate in the frequency domain, and $h_{\xi}$ is the mesh size defined by $h_{\xi} = N/(dM)$. Due to the symmetry property $y(\theta,\kappa) = y(-\theta,-\kappa)$ for exact data $y$, it follows that $y(\frac{j\pi}{p} + \pi,\frac{l\varrho}{q}) = y(\frac{j\pi}{p},-\frac{l\varrho}{q})$, allowing us to extend the observation to $\pi \leq \phi \leq 2\pi$. With this, the value of $\F_2 y^{\delta}(\frac{\xi}{\abs{\xi}},\abs{\xi})$ can then be approximated by
    \begin{equation*}
    \begin{split}
        \F_2 y^{\delta}\kl{\frac{\xi_k}{\abs{\xi_k}}, \abs{\xi_k}}
        =
        \int_{\R} e^{-i \abs{\xi_k} \kappa} y^{\delta}\kl{ \frac{\xi_k}{\abs{\xi_k}},\kappa} \, d\kappa
        \approx
        \frac{\varrho}{q}
        \sum_{l=-q}^{q} e^{-i\abs{\xi_k}\kappa_l} y^{\delta}\kl{\frac{\xi_k}{\abs{\xi_k}},\kappa_l} \,,
    \end{split}
    \end{equation*}
for $k\neq\mfz $. The value of $y^{\delta}(\frac{\xi_k}{\abs{\xi_k}},\kappa_l)$ for $l = -q\,, \dots \,, q$ can be evaluated by linear interpolation as follows: For each Cartesian grid point $k = (k_1,k_2) \in \Z^2$ with $k_I = -dM\,, \dots \,, dM (I=1,2)$, we can estimate $y^{\delta}(\frac{\xi_k}{\abs{\xi_k}},\kappa_l)$ by the following steps:
\begin{enumerate}
    \item Convert $\xi_k$ into the polar coordinates $(\phi_k',\kappa_k')$.
    \item Compute $K = \lfloor\phi_k'p / \pi\rfloor$. The two beams next to $\xi_k$ then are
        \begin{equation*}
            (\cos(K\pi/p),\sin(K\pi/p))\,,
            \qquad
            \text{and}
            \qquad
            (\cos((K+1)\pi/p),\sin((K+1)\pi/p) \,.
        \end{equation*}
    \item Compute the linear interpolation coefficients $a_{k}, b_{k}$ via
        \begin{equation*}
            a_{k} = (\kl{K+1} \pi / p-\phi_k') \cdot p / \pi \,,
            \qquad
            \text{and}
            \qquad
            b_{k} = (\phi_k'-K\pi / p) \cdot p / \pi \,.
        \end{equation*}
    \item Estimate $y^{\delta}(\frac{\xi_k}{\abs{\xi_k}},\kappa_l)$ by $a_k y^{\delta}(\theta_K,\kappa_l) + b_k y^{\delta}(\theta_{K+1},\kappa_l)$.
\end{enumerate}
On the other hand, if $k = \mfz$, due to \eqref{eq_RF} for any $\theta\in\mathbb{S}$ there holds $\F_2 y(\theta,0)=\F f(\mfz)$. The numerical evaluation of the term $\F_2 y^{\delta}(-\frac{\xi_k}{\abs{\xi_k}},-\abs{\xi_k})$ can be done analogously to the above steps by setting $y^{\delta}(\frac{j\pi}{p} + \pi,\frac{l\varrho}{q}) = y^{\delta}(\frac{j\pi}{p},-\frac{l\varrho}{q})$. After these preliminaries we can now state our main algorithm.

\begin{algorithm}\label{algo}
Let $y_{j,l} = y^{\delta}(\theta_j,s_l)$ be given for $j=0\,, \dots \,, p-1$, and $l=-q\,,\dots\,,q$. Then the estimate $\fam$ of the original image $f$ at the points $(m_1h_x,m_2h_x)$ for $m\in\Z^2$ and $m_I=-M,\dots,M(I=1,2)$ is computed by
\begin{itemize}
    \item \emph{Step 1:} Extend $y_{j,l}$ to the case $j=0,\dots,2p-1$ and $l=-q,\dots,q$ by setting $y_{j+p,l} = y_{j,-l}$. For $k_I = -dM \,, \dots \,, dM(I=1,2)$ and $k \neq \mfz $, we compute
        \begin{equation*}
            \F_2 y_k : =
            \frac{\varrho}{q}
            \sum_{l=-q}^{q} e^{-i \abs{\xi_k} \kappa_l}(a_k y^{\delta}(\theta_K,\kappa_l) + b_k y^{\delta}(\theta_{K+1},\kappa_l)) \,.
        \end{equation*}
    For $k = \mfz $, we compute $\F_2 y_0=\frac{1}{p}\sum_{j=0}^{p-1}\F_2 y^{\delta}(\theta_j ,0)$. Here, if $K+1 = 360$, we substitute $y^{\delta}(\theta_{K+1},\kappa_l)$ by $y^{\delta}(\mfz ,\kappa_l)$. Further, we obtain $\F_2 y_{-k}$ as an estimate of $\F_2 y(-\frac{\xi}{\abs{\xi}},-\abs{\xi})$ analogously to the above approach. One may implement the algorithm in \cite{Fourmont_2003} to accelerate this step.
    \item \emph{Step 2:} Recalling \eqref{A2}, we compute a discrete approximation $\F \fak$ of $\F \fa$ via
        \begin{equation*}
            \F \fak :=
            \frac{\F_2 y_k + \F_2 y_{-k}}{2 + \alpha(2\pi)^{1-n}\abs{\xi_k}^{n-1}(1+\abs{\xi_k}^2)^s} \,.
        \end{equation*}
    \item \emph{Step 3:} Compute the inverse Fourier transform of $\F \fak,k\in\Z^2,k_I=-dM,\dots,dM(I=1,2)$ by truncating the frequency domain and computing its discrete form, i.e.,
        \begin{equation*}
        \begin{split}
        	f_{\alpha,m}=\left(\frac{N}{dM}\right)^2\left(\frac{1}{2\pi}\right)^2\sum_{k_1,k_2=-dM}^{dM}e^{ix_m\xi_k}\mathcal{F}f_{\alpha,k} ,\\m=(m_1,m_2)\in\mathbb{Z}^2,m_I=-M,\cdots,M(I=1,2),
        \end{split}
        \end{equation*}
    for which we can use a standard inverse FFT. Then, $f_{\alpha,m}$ is the estimate of $\fad$ at the point $(m_1h_x,m_2h_x),h_x=\tau/M$.
\end{itemize}
\end{algorithm}

\subsection{Numerical examples}\label{subse4_2}

We provide several numerical examples verifying the efficiency of our proposed Fourier-based Tikhonov regularization method \eqref{re}, i.e., Algorithm~\ref{algo} in this subsection. As test models we use the dataset measured from a carved cheese and a walnut provided by the Finish Inverse Problems Community, see \cite{bubba2017tomographic,hamalainen2015tomographic} for further details.

\begin{example}[Reconstruction with downsampling.]
    \begin{figure}[ht!]
        \centering
        \includegraphics[width=16cm]{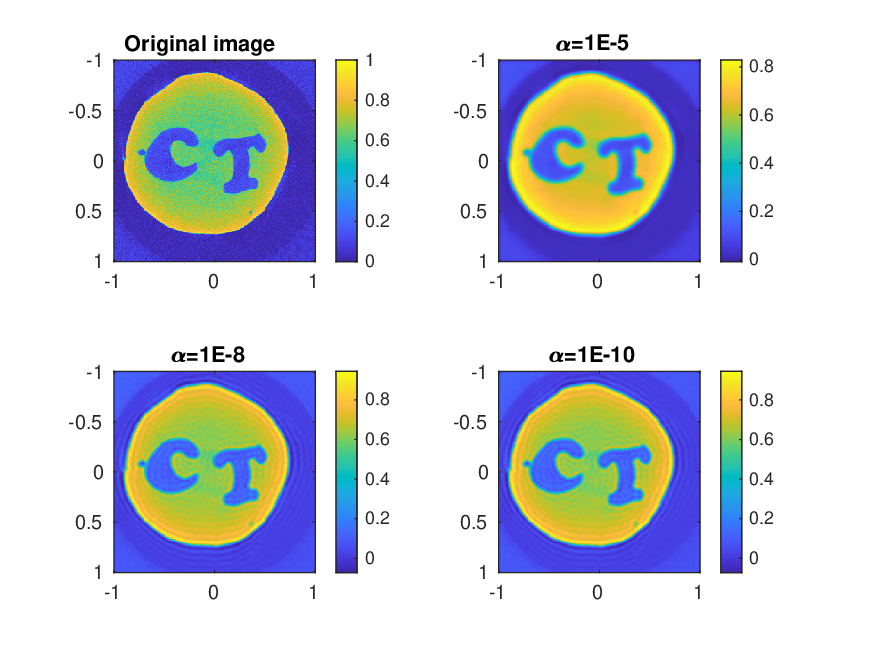}
        \caption{Original image of carved cheese (upper left) and images reconstructed with the Fourier-based Tikhonov regularization method \eqref{re} with $\alpha=10^{-5},10^{-8},10^{-10}$.}
        \label{nonoise}
    \end{figure}
    \begin{table}[ht!]\centering
		\begin{tabular}{|c|c|c|c|c|c|c|}
            \hline     & $\alpha=10^{-5}$ & $\alpha=10^{-8}$ & $\alpha=10^{-10}$ \\
			\hline MSE     &0.0050& 0.0034 &0.0034            \\
				\hline PSNR    &23.0213& 24.7253 &24.7257            \\
                \hline SSIM     &0.4888& 0.4966&0.4963                  \\
				\hline
		\end{tabular}
		\caption{Comparison of average MSE, SSIM and PSNR for the Fourier-based Tikhonov regularization method \eqref{re} with downsampling measurement for the carved cheese.}\label{tabnonoise}
    \end{table}

    \begin{figure}[ht!]
        \centering
        \includegraphics[width=16cm]{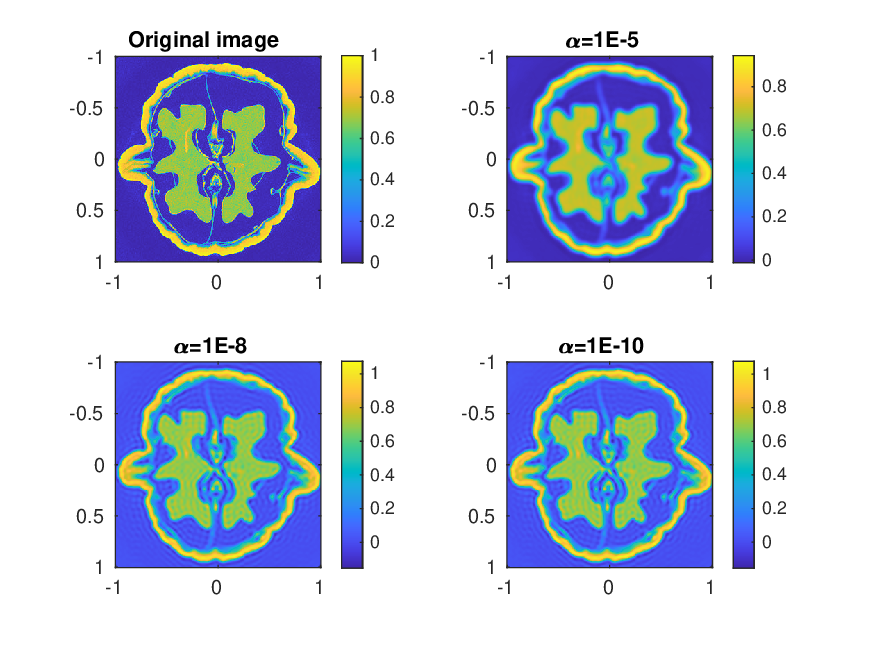}
        \caption{Original image of walnut (upper left) and images reconstructed with the Fourier-based Tikhonov regularization method \eqref{re} with $\alpha=10^{-5},10^{-8},10^{-10}$.}
        \label{nonoisewal}
    \end{figure}

    \begin{table}[ht!]\centering
        \begin{tabular}{|c|c|c|c|c|c|c|}
            \hline     & $\alpha=10^{-5}$ & $\alpha=10^{-8}$ & $\alpha=10^{-10}$ \\
			\hline MSE  & 0.0166 & 0.0119 & 0.0119 \\
			\hline PSNR & 17.8070 & 19.2459 & 19.2466 \\
            \hline SSIM & 0.5834 & 0.5775 & 0.5766 \\
			\hline
		\end{tabular}
		\caption{Comparison of average MSE, SSIM and PSNR for the Fourier-based Tikhonov regularization method \eqref{re} with downsampling measurement for the walnut.}\label{tabnonoisewal}
    \end{table}

In the first experiment, we took the high resolution CT measurements of the carved cheese image and the walnut image, and downsampled it from a dimension of $2000\times 2000$ to $300\times 300$ and a dimension of $2296\times 2296$ to $300\times300$, respectively. The proposed reconstruction approach is \eqref{re} or Algorithm~\ref{algo}. Due to the averaging during the downsampling, we expect the data being only affected with little noise. For the embedding operator we choose  $s=1.2$. Figures~\ref{nonoise} and \ref{nonoisewal} depict the original image as well as our reconstructed images for different small values of the regularization parameter $\alpha$. A quantitative comparison of the different reconstructions via the average MSE (Mean Square Error), PSNR (Peak Signal to Noise Ratio), and SSIM (Structured Similarity Indexing Method) for different parameters $\alpha$ is given in Tables \ref{tabnonoise} and \ref{tabnonoisewal}. As can be seen from the reconstructed images and the quantitative comparison, our reconstruction algorithm performs well in particular for small values of $\alpha$, which is to be expected in case of noise free measurements.
\end{example}

\begin{example}[Reconstruction with additional irregular noise]\label{example_2}

In the second experiment, we investigate the behaviour of our Fourier-based Tikhonov regularization method with data where we included additional noise in the measurements. To validate the recovery accuracy, we first focus on a comparison of the best reconstructions obtained with three different methods. In particular, we compare our  proposed method \eqref{re} for the preprocessed problem with FBP, and Tikhonov regularization with a total variation (TV) penalty for the original inverse problem, i.e.,
    \begin{equation}\label{eq_TVregu}
        \fd_{TV}:= {\arg\min}_f \Kl{ \norm{R f}_{\LtSR}^2 - 2 \spr{ Rf, \yd }_{\LtSR} + \alpha \norm{ f }_{TV} } \,.
    \end{equation}
To compare the performance of different reconstruction algorithms, we again consider the carved cheese and walnut samples, to which we now add Gaussian noise, i.e., $y^\delta(\theta,\kappa):= y(\theta,\kappa) + \delta \cdot \max|y(\theta,\kappa)| \cdot \mathcal{N}(0,1)$ with $\delta=0.3$, $0.5$ respectively. In order to minimize the influence of randomness on the comparison, we apply the reconstruction methods to $10$ different tests with the same noise level, and record the average MSE, PSNR, and SSIM in Tables~\ref{ex2_table1}-\ref{ex2_table2}. As can be observed, the FBP algorithm does not provide reasonable reconstruction in case of additional noise. On the other hand, our proposed method yields comparable results to those obtained via Tikhonov regularization with a TV penalty. In particular for a large noise level, our proposed method seems to be more robust than the other two.

    \begin{table}[ht!]\centering
		\begin{tabular}{|c|c|c|c|c|c|c|}
            \hline      & Filtered Backprojection & Total Variation & Our method\\
			\hline MSE     &4.4431& 0.0217 & 0.0155   \\
			\hline PSNR    &13.6359& 16.6408 & 18.0989           \\
            \hline SSIM     &0.0010& 0.2568&0.2798                  \\
			\hline
		  \end{tabular} \\
		\begin{tabular}{|c|c|c|c|c|c|c|}
            \hline      & Filtered Backprojection & Total Variation & Our method \\
			\hline MSE     &5.5673& 0.0466 &  0.0351          \\
			\hline PSNR    &13.1493& 13.5046 & 14.4980           \\
            \hline SSIM     &0.0023& 0.2030 & 0.2001                 \\
			\hline
		\end{tabular}
		\caption{MSE, SSIM, and PSNR for the reconstruction methods with $\delta=0.3$ averaged over $10$ different tests. Carved cheese sample (top) and walnut sample (bottom table).}\label{ex2_table1}
    \end{table}


    \begin{table}[ht!]\centering
        \begin{tabular}{|c|c|c|c|c|c|c|}
            \hline      & Filtered Backprojection & Total Variation & Our method \\
			\hline MSE     &11.1273& 0.0390  & 0.0230          \\
			\hline PSNR    &13.5498& 14.0904 &  16.3995          \\
            \hline SSIM     &$2.9988\cdot10^{-4}$& 0.1920  & 0.2751               \\
			\hline
		\end{tabular}
        \\
		\begin{tabular}{|c|c|c|c|c|c|c|}
            \hline      & Filtered Backprojection & Total Variation & Our method\\
			\hline MSE     &12.9145& 0.0718  & 0.0479          \\
			\hline PSNR    &13.4330& 11.4367 & 13.1949           \\
            \hline SSIM     &$6.9233\cdot10^{-4}$& 0.1436 &  0.1904                 \\
			\hline
		\end{tabular}
		\caption{MSE, SSIM, and PSNR for the reconstruction methods with $\delta=0.5$ averaged over $10$ different tests. Carved cheese sample (top) and walnut sample (bottom table).}\label{ex2_table2}		
    \end{table}

\end{example}

It is worth to note that our proposed approach yields a well-defined discrepancy term $\norm{\Rt f - \gd}_Y^2$, even though the actual values may be large due to the random noise as in the above Example \ref{example_2}. This allows us to consider discrepancy-based \emph{a-posteriori} parameter choice rules such as the modified discrepancy principle \cite{Blanchard_Mathe_2012,Lu_Mathe_2014} or heuristic parameter choices rules \cite[Ch.4]{Engl_Hanke_Neubauer_1996}. Here, we consider the modified L-curve method \cite{Reginska_1996,Lu_Mathe_2013}, which determines the \emph{a-posteriori} regularization parameter $\alpha_*$ via the minimization problem
    \begin{equation}\label{eq_Lcurve1}
        \alpha_{*}= \min_{\alpha} J_F(\alpha) \,,
        \qquad
        \text{where}
        \qquad
        J_F(\alpha):= \norm{ \Rt \fad - \gd }^2_{\LtSR} \norm{ \fad}_{\LtRt}^2 \,.
    \end{equation}
Note that this heuristic parameter choice rule cannot be implemented directly for Tikhonov regularization with TV penalty given in \eqref{eq_TVregu}, since the discrepancy there is not well-defined due to the irregular noise. However, it is possible to consider Tikhonov regularization with a TV penalty for the preprocessed problem, i.e.,
    \begin{equation}\label{eq_TVregu2}
        \ftdTV:= {\arg\min}_{\ft} \Kl{ \norm{\Rt \ft -\gd}_Y^2 + \alpha\norm{\ft}_{TV} } \,.
    \end{equation}
Even though a theoretical analysis for the above approach is lacking, one can implement this method together with the modified L-curve method
    \begin{equation}\label{eq_Lcurve2}
        \alpha_{*,TV} := \min_{\alpha} \JTV(\alpha) \,,
        \qquad
        \text{where}
        \qquad
        \JTV(\alpha) := \norm{\Rt \fdTV - \gd}^2_{\LtSR} \norm{\fdTV}_{\LtRt} \,.
    \end{equation}
Numerical results when using this approach are given in the following.

\begin{example}\label{example_3}
In the last example, we again consider the carved cheese and walnut cases and test the performance of the modified L-curve method when our proposed method and the Tikhonov regularization with total variation (\ref{eq_TVregu2}) are implemented to solve the preprocessed problem
    \begin{equation*}
        \Rt f = \gd \,.
    \end{equation*}
To this end, we choose two regularization parameter sets for both approaches and calculate the minimizers of both modified L-curve methods (\ref{eq_Lcurve1}) and (\ref{eq_Lcurve2}). To reduce the influence of the randomness, we again test these examples $10$ times and provide the mean values of the MSE, PSNR, SSIM in Table~\ref{ex3_table1} and Table~\ref{ex3_table3}, respectively. Similarly as before, we observed that our proposed method is more robust than the approach (\ref{eq_TVregu2}) under the heuristic modified L-curve methods (\ref{eq_Lcurve1}) and (\ref{eq_Lcurve2}).
    \begin{table}[ht!]\centering
		\begin{tabular}{|c|c|c|c|c|}
            \hline       &\multicolumn{2}{|c|}{$\delta=0.3$}&\multicolumn{2}{|c|}{$\delta=0.5$}\\
            \cline{2-5}           &Our method&Total Variation&Our method&Total Variation\\
			\hline MSE     &0.0416  &0.0311  &0.0462 & 0.0626          \\
			\hline PSNR    &13.8061    &15.0717  &13.3532 & 12.1634           \\
            \hline SSIM     &0.2741    &0.2482  &0.2582 & 0.2153                  \\
			\hline
		\end{tabular}
		\caption{Comparison of average MSE, SSIM and PSNR for the carved cheese under modified L-curve parameter choice rule.}\label{ex3_table1}
	\end{table}
\begin{table}[ht!]\centering
		\begin{tabular}{|c|c|c|c|c|}
            \hline       &\multicolumn{2}{|c|}{$\delta=0.3$}&\multicolumn{2}{|c|}{$\delta=0.5$}\\
            \cline{2-5}           &Our method&Total Variation&Our method&Total Variation\\
			\hline MSE     &0.0608  &0.0871  &0.0901 & 0.1252         \\
			\hline PSNR    &12.1631    &10.5987  &10.4553 & 9.0217           \\
            \hline SSIM     &0.2505    &0.1512  &0.1636 & 0.1116                  \\
			\hline
		\end{tabular}
		\caption{Comparison of average MSE, SSIM and PSNR for the walnut under modified L-curve parameter choice rule.}\label{ex3_table3}
	\end{table}
\end{example}

\section*{Acknowledgement}
This work is supported by Key-Area Research and Development Program of Guangdong Province (No.2021B0101190003). This research was funded in part by the Austrian Science Fund (FWF) SFB 10.55776/F68 ``Tomography Across the Scales'', project F6805-N36 (Tomography in Astronomy). For open access purposes, the authors have applied a CC BY public copyright license to any author-accepted manuscript version arising from this submission. S.~Lu is supported by NSFC (No.11925104), and the Sino-German Mobility Programme (M-0187) by Sino-German Center for Research Promotion.

Part of the results are done when the first author visited the fourth author at the Johann Radon Institute for Computational and Applied Mathematics (RICAM) in 2021. She would like to thank him for the invitation and kind hospitality.

\bibliographystyle{plain}
{\footnotesize
\bibliography{mybib}

\begin{thebibliography}{10}

\bibitem{Adams_1970}
R.~A. Adams.
\newblock {Equivalent norms for Sobolev spaces}.
\newblock {\em Proc. Amer. Math. Soc.}, 24:63--66, 1970.

\bibitem{Adams_Fournier_2003}
R.~A. Adams and J.~J.~F. Fournier.
\newblock {\em Sobolev Spaces}.
\newblock Pure and Applied Mathematics. Elsevier Science, 2003.

\bibitem{Andersen_Kak_1984}
A.~H. Andersen and A.~C. Kak.
\newblock Simultaneous algebraic reconstruction technique (sart): a superior
  implementation of the art algorithm.
\newblock {\em Ultrasonic imaging}, 6(1):81--94, 1984.

\bibitem{Blanchard_Mathe_2012}
G.~Blanchard and P.~Math\'{e}.
\newblock Discrepancy principle for statistical inverse problems with
  application to conjugate gradient regularization.
\newblock {\em Inverse PRoblems}, 28(11):115011, 2012.

\bibitem{bubba2017tomographic}
Tatiana~A Bubba, Markus Juvonen, Jonatan Lehtonen, Maximilian M{\"a}rz,
  Alexander Meaney, Zenith Purisha, and Samuli Siltanen.
\newblock Tomographic x-ray data of carved cheese.
\newblock {\em arXiv preprint arXiv:1705.05732}, 2017.

\bibitem{Cavalier_2011}
L.~Cavalier.
\newblock {\em Inverse Problems in Statistics}.
\newblock Springer Berlin Heidelberg, Berlin, Heidelberg, 2011.

\bibitem{chang1980scientific}
Tao Chang and Gabor~T Herman.
\newblock A scientific study of filter selection for a fan-beam convolution
  reconstruction algorithm.
\newblock {\em SIAM Journal on Applied Mathematics}, 39(1):83--105, 1980.

\bibitem{Cheung_Lewitt_1991}
W.~K. Cheung and R.~M. Lewitt.
\newblock {Modified Fourier reconstruction method using shifted transform
  samples}.
\newblock {\em Physics in medicine \& biology}, 36(2):269--277, 1991.

\bibitem{Douglas_1996}
R.~Douglas.
\newblock {On majorization, factorization, and range inclusion of operators on
  Hilbert space}.
\newblock {\em Proceedings of the American Mathematical Society},
  17(2):413--415, 1966.

\bibitem{Engl_Hanke_Neubauer_1996}
H.~W. {Engl}, M.~{Hanke}, and A.~{Neubauer}.
\newblock {\em {Regularization of inverse problems.}}
\newblock Dordrecht: Kluwer Academic Publishers, 1996.

\bibitem{Fourmont_2003}
K.~Fourmont.
\newblock Non-equispaced fast fourier transforms with applications to
  tomography.
\newblock {\em Journal of Fourier Analysis and Applications}, 9(5):431--450,
  2003.

\bibitem{Guan_Gordon_1996}
H.~Guan and R.~Gordon.
\newblock {Computed tomography using algebraic reconstruction techniques (ARTs)
  with different projection access schemes: a comparison study under practical
  situations}.
\newblock {\em Physics in Medicine \& Biology}, 41(9):1727--1743, 1996.

\bibitem{hamalainen2015tomographic}
Keijo H{\"a}m{\"a}l{\"a}inen, Lauri Harhanen, Aki Kallonen, Antti
  Kujanp{\"a}{\"a}, Esa Niemi, and Samuli Siltanen.
\newblock Tomographic x-ray data of a walnut.
\newblock {\em arXiv preprint arXiv:1502.04064}, 2015.

\bibitem{Hohage_Werner_2013}
T.~Hohage and F.~Werner.
\newblock {Iteratively regularized Newton-type methods for general data misfit
  functionals and applications to Poisson data}.
\newblock {\em Numerische Mathematik}, 13(4):745--779, 2013.

\bibitem{Hubmer_Sherina_Ramlau_2023}
S.~Hubmer, E.~Sherina, and R.~Ramlau.
\newblock {Characterizations of Adjoint Sobolev Embedding Operators with
  Applications in Inverse Problems}.
\newblock {\em Electronic Transactions on Numerical Analysis}, 59:116--144,
  2023.

\bibitem{Kekkonen_Lassas_Siltanen_2014}
H.~Kekkonen, M.~Lassas, and S.~Siltanen.
\newblock {Analysis of regularized inversion of data corrupted by white
  Gaussian noise}.
\newblock {\em Inverse Problems}, 30(4):045009, 2014.

\bibitem{Klann_Maass_Ramlau_2006}
E.~Klann, P.~Maass, and R.~Ramlau.
\newblock Two-step regularization methods for linear inverse problems.
\newblock {\em Journal of Inverse and Ill-Posed Problems}, 14(6):583--607,
  2006.

\bibitem{Klann_Ramlau_2008}
E.~Klann and R.~Ramlau.
\newblock Regularization by fractional filter methods and data smoothing.
\newblock {\em Inverse Problems}, 24(2), 2008.

\bibitem{Lin_Lu_Mathe_2015}
K.~Lin, S.~Lu, and P.~Mathe.
\newblock {Oracle-type posterior contraction rates in Bayesian inverse
  problems}.
\newblock {\em Inverse Problems and Imaging}, 9(3):895--915, 2015.

\bibitem{Lu_Mathe_2013}
S.~Lu and P.~Mathe.
\newblock Heuristic parameter selection based on functional minimization:
  optimality and model function approach.
\newblock {\em Mathematics of Computation}, 82(283):1609--1630, 2013.

\bibitem{Lu_Mathe_2014}
S.~Lu and P.~Math\'{e}.
\newblock Discrepancy based model selection in statistical inverse problems.
\newblock {\em Journal of Complexity}, 30(3):290--308, 2014.

\bibitem{Lu_Pereverzev_2013}
S.~Lu and S.~V. Pereverzev.
\newblock {\em Regularization Theory for Ill-posed Problems}.
\newblock De Gruyter, Berlin, Boston, 2013.

\bibitem{Mathe_2019}
P.~Math\'{e}.
\newblock Bayesian inverse problems with non-commuting operators.
\newblock {\em Mathematics of Computation}, 88(320):2897--2912, 2019.

\bibitem{Mathe_Pereverzev_2003}
P.~Math\'{e} and S.~V. Pereverzev.
\newblock Inverse problems.
\newblock {\em Geometry of linear ill-posed problems in variable Hilbert
  scales}, 19(3):789--803, 2003.

\bibitem{Mathe_Tautenhahn_2006}
P.~Math\'{e} and U.~Tautenhahn.
\newblock Interpolation in variable hilbert scales with applications to inverse
  problems.
\newblock {\em Inverse Problems}, 22(6):2271--2297, 2006.

\bibitem{MR2855055}
Peter Math\'{e} and Ulrich Tautenhahn.
\newblock Enhancing linear regularization to treat large noise.
\newblock {\em J. Inverse Ill-Posed Probl.}, 19(6):859--879, 2011.

\bibitem{MR2772535}
Peter Math\'{e} and Ulrich Tautenhahn.
\newblock Regularization under general noise assumptions.
\newblock {\em Inverse Problems}, 27(3):035016, 15, 2011.

\bibitem{McLean_2000}
W.~C.~H. McLean.
\newblock {\em Strongly Elliptic Systems and Boundary Integral Equations}.
\newblock Cambridge University Press, 2000.

\bibitem{Natterer_1985}
F.~Natterer.
\newblock Fourier reconstruction in tomography.
\newblock {\em Numerische Mathematik}, 47(3):343--353, 1985.

\bibitem{Natterer_2001}
F.~{Natterer}.
\newblock {\em {The Mathematics of Computerized Tomography}}.
\newblock Society for Industrial and Applied Mathematics, Philadelphia, PA,
  2001.

\bibitem{Ramlau_Teschke_2004_1}
R.~Ramlau and G.~Teschke.
\newblock {Regularization of Sobolev Embedding Operators and Applications to
  Medical Imaging and Meteorological Data. Part I: Regularization of Sobolev
  Embedding Operators}.
\newblock {\em Sampling Theory in Signal and Image Processing}, 3(2):175--195,
  2004.

\bibitem{Ramlau_Teschke_2004_2}
R.~Ramlau and G.~Teschke.
\newblock {Regularization of Sobolev Embedding Operators and Applications to
  Medical Imaging and Meteorological Data. Part II: Regularization
  Incorporating Noise with Applications in Medical Imaging and Meteorological
  Data}.
\newblock {\em Sampling Theory in Signal and Image Processing}, 3(3):205--226,
  2004.

\bibitem{Reginska_1996}
T.~Reginska.
\newblock A regularization parameter in discrete ill-posed problems.
\newblock {\em SIAM Journal on Scientific Computing}, 17(3):740--749, 1996.

\bibitem{Reynolds_Matthew_Beylkin_Monzan_2013}
M.~Reynolds, Matthew~G. Beylkin, and L.~Monzan.
\newblock Rational approximations for tomographic reconstructions.
\newblock {\em Inverse Problems}, 29(6):065020, 2013.

\bibitem{Sazonov_1958}
V.~Sazonov.
\newblock {A Remark on Characteristic Functionals}.
\newblock {\em Theory of Probability and its Applications}, 3(2):188--192,
  1958.

\bibitem{Schomberg_Timmer_1995}
H.~Schomberg and J.~Timmer.
\newblock {The gridding method for image reconstruction by Fourier
  transformation}.
\newblock {\em IEEE transactions on medical imaging}, 14(3):596--607, 1995.

\bibitem{shepp1974fourier}
Lawrence~A Shepp and Benjamin~F Logan.
\newblock The fourier reconstruction of a head section.
\newblock {\em IEEE Transactions on nuclear science}, 21(3):21--43, 1974.

\bibitem{smith1985image}
Bruce~D Smith.
\newblock Image reconstruction from cone-beam projections: necessary and
  sufficient conditions and reconstruction methods.
\newblock {\em IEEE transactions on medical imaging}, 4(1):14--25, 1985.

\bibitem{smith1985mathematical}
Kennan~T Smith and F~Keinert.
\newblock Mathematical foundations of computed tomography.
\newblock {\em Applied Optics}, 24(23):3950--3957, 1985.

\bibitem{Xu_Liow_Strother_1993}
X.~Xu, J.~S. Liow, and S.~C. Strother.
\newblock Iterative algebraic reconstruction algorithms for emission computed
  tomography: A unified framework and its application to positron emission
  tomography.
\newblock {\em Medical physics}, 20(6):1675--1684, 1993.

\end{thebibliography}
}

\end{document}